\documentclass[11pt,a4paper]{article}
\usepackage{amssymb,amsmath,amsfonts,mathrsfs,bm,enumerate}
\usepackage{xcolor}
\allowdisplaybreaks[1]
\numberwithin{equation}{section}
\usepackage{comment}

\usepackage[colorlinks=true, pdfstartview=FitV, linkcolor=blue, citecolor=blue, urlcolor=blue,pagebackref=false]{hyperref}

\usepackage[normalem]{ulem}
\usepackage{tikz}
\usepackage{float}
\usepackage[font=footnotesize]{caption}

\usepackage{times,theorem,latexsym}

\newcommand{\BOX}{\ensuremath\Box}

\newtheorem{theorem}{Theorem }[section]

\newtheorem{definition}[theorem]{Definition}
{\theorembodyfont{\rmfamily}}
{\theorembodyfont{\rmfamily}}
{\theorembodyfont{\rmfamily}}
\newtheorem{lemma}[theorem]{Lemma}
\newtheorem{proposition}[theorem]{Proposition}
{\theorembodyfont{\rmfamily}\newtheorem{remark}[theorem]{Remark}}
{\theorembodyfont{\rmfamily}}

\newcommand{\N}{\mathbb{N}}
\newcommand{\Z}{\mathbb{Z}}
\newcommand{\R}{\mathbb{R}}

\newcommand{\T}{\mathbb{T}}

\newcommand{\dd}{\,d}

\newcommand{\opspan}{\operatorname{span}}

\newcommand{\opN}{\operatorname{\mathit{N}}}
\newcommand{\opR}{\operatorname{\mathit{R}}}
\newcommand{\opcodim}{\operatorname{codim}}

\newcommand{\eps}{\varepsilon}
\newcommand{\del}{\delta}

\newcommand{\opdist}{\operatorname{dist}}

\newcommand{\bv}[1]{\bm{#1}}

\newcommand{\overbar}[1]{\mkern 1.5mu\overline{\mkern-1.5mu#1\mkern-1.5mu}\mkern 1.5mu}

\def\XXint#1#2#3{{\setbox0=\hbox{$#1{#2#3}{\int}$}
		\vcenter{\hbox{$#2#3$}}\kern-.5\wd0}}

\newenvironment{proof}{{\vskip\baselineskip\noindent\textbf{Proof:}}}
{\hspace*{.1pt}\hspace*{\fill}\BOX\vskip\baselineskip}

\begin{document}

\title{
Lyapunov Unstable Motion Bifurcating\\ from a Circular Vortex Filament
}

\author{
Masashi Aiki
\thanks{
Department of Mathematics, 
Faculty of Science and Technology, 
Tokyo University of Science,
2641 Yamazaki, Noda, Chiba 278-8510 Japan.
\textit{E-mail address:} \texttt{a27120@rs.tus.ac.jp}
}
\and
Mitsuo Higaki
\thanks{
Department of Mathematics, 
Graduate School of Science, 
Kobe University, 
1-1 Rokkodai, Nada-ku, Kobe 657-8501, Japan.
\textit{E-mail address:} \texttt{higaki@math.kobe-u.ac.jp}
}
}
\date{}

\maketitle

\begin{abstract}
This paper investigates the dynamics of closed vortex filaments in $\R^3$ governed by the Localized Induction Equation. Recently, Aiki and Higaki \cite{AikiHigaki2026} established the nonlinear orbital stability of circular vortex filaments under asymmetric perturbations, while identifying Lyapunov instability due to the linear growth of translation modes. Motivated by this result, we prove the existence of a family of closed solutions, which we call axial screw motions, that bifurcate from a circular filament. 
These solutions remain uniformly close to the orbit of the circle, but drift secularly away from the reference motion because their translation speed along the symmetry axis differs from that of the circular filament. In particular, they provide explicit non-trivial perturbations that satisfy orbital-stability estimates while failing Lyapunov stability, thereby realizing the gap between orbital stability and Lyapunov stability near the circular filament.
\end{abstract}

\tableofcontents

    \section{Introduction}
    \label{sec.intro}

The motion of a vortex ring has been a major topic in fluid mechanics for a very long time, dating back to the seminal paper by Helmholtz \cite{Helmholtz}. A vortex ring is a torus-shaped region in a fluid in which the vorticity is concentrated. Such objects are abundant in our surroundings, for example as bubble rings produced by dolphins or as air vortex rings generated by jet engines. The dynamics and stability of vortex rings have been extensively studied through experiments, numerical simulations, and theoretical analyses by many researchers 
\cite{Fraenkel,KamebeTakao,Norbury,Maxworthy72,Maxworthy77,Sullivan,WidnallSullivan,FraenkelBerger,WidnallBlissTsai,Saffman,WidnallTsai,Kida1982,GigaMiyakawa,KnioGhoniem,ShariffLeonard,ShariffVerziccoOrlandi,Fukumoto,FukumotoHattori2002,FukumotoHattori2005,HattoriFukumoto,FengSverak,HattoriRodriguezDizes,Protas,Akinshin,BalakrishnaMathewSamanta,ChoiJeong,CaoQinZhanZou,Choi}. 
Despite this large body of work, rigorous mathematical results remain comparatively limited. In the framework of the Euler equations, Choi and Jeong \cite{ChoiJeong} and Choi \cite{Choi} proved orbital stability of a vortex ring under symmetric perturbations. In that Euler setting, orbital stability means stability up to translations along the axis of symmetry.

When a vortex ring is sufficiently thin, it is well known that its motion can be approximated by that of its centerline, namely a circular vortex filament. Generally, a vortex filament is an idealized space curve on which the vorticity of the fluid is concentrated, and the vorticity at each point of the curve is pointed in the direction of the tangent vector of the curve at that point. One of the most fundamental equations modeling the motion of a vortex filament is the Localized Induction Equation, given by
\begin{equation*}
    \bv{x}_t 
    = \bv{x}_s \times \bv{x}_{ss}.
\end{equation*}
Here, $\bv{x}=\bv{x}(s,t)$ is the position vector of the filament in $\R^3$, parametrized by arc length $s$ at time $t$; subscripts denote partial derivatives, and $\times$ denotes the exterior product in $\R^3$.

The Localized Induction Equation was derived independently by Da Rios \cite{DaRios}, Murakami et al. \cite{Murakamietal}, and Arms and Hama \cite{ArmsHama1965}. They derived the equation by approximating the Biot-Savart integral when the vorticity distribution is given as a vortex filament. Although this derivation is somewhat formal, Jerrard and Seis \cite{JerrardSeis} established a rigorous conditional result: under suitable energy and regularity assumptions, if the vorticity remains concentrated around a closed curve, then the leading-order curve dynamics, after rescaling time, are governed by the Localized Induction Equation. Hence, a detailed analysis of the Localized Induction Equation has possible implications for the analysis of the Euler equations.

To formulate our problem for the circular vortex filament, we fix a radius $R>0$ and consider the periodic problem describing the motion of a closed filament, given by
\begin{equation}\tag{LIE}\label{eq.LIE}
    \bv{x}_t 
    = \bv{x}_s \times \bv{x}_{ss}
    \quad 
    \text{in} \mkern9mu \T \times (0,\infty). 
\end{equation}
Here, $\T = \R /L\Z $ and $L=2\pi R $. The equation \eqref{eq.LIE} describes the motion of a closed filament with length $L$. Let $\bv{x}^R_0$ be the stationary circular profile given by
\[
    \bv{x}^R_0(s) 
    = 
    R \cos\Big(\frac{s}{R}\Big) \bv{e}_1 
    + R \sin\Big(\frac{s}{R}\Big) \bv{e}_2,
    \quad
    s\in\T, 
\]
where $\{\bv{e}_1, \bv{e}_2, \bv{e}_3\}$ is the canonical basis of $\R^3$. We further define $\bv{x}^R$ by
\[
    \bv{x}^R(s,t) = \bv{x}^R_0(s) + \frac{t}{R} \bv{e}_3,
    \quad
    (s,t)\in\T\times[0,\infty).
\]
It is straightforward to check that $\bv{x}^{R}$ satisfies \eqref{eq.LIE}. The stability of the circular filament $\bv{x}^{R}$ is the main topic of this paper.

Before we proceed further, we first go over related results considering \eqref{eq.LIE}. Nishiyama and Tani \cite{NishiyamaTani1994,NishiyamaTani1996,TaniNishiyama1997} proved global-in-time solvability of the initial value problem for \eqref{eq.LIE} in Sobolev spaces. Koiso \cite{Koiso1997} considered a geometrically generalized setting and rigorously proved the equivalence between the solvability of the initial value problem for \eqref{eq.LIE} and that for the cubic nonlinear Schr\"odinger equation.

Regarding the stability of circular vortex filaments, Calini, Keith, and Lafortune \cite{CaliniKeithLafortune2011} and Ivey and Lafortune \cite{IveyLafortune} developed techniques and utilized them to prove that closed vortex filaments are stable under the linearized Localized Induction Equation.

Jerrard and Smets \cite{JerrardSmets2012,JerrardSmets2015} proposed a weak formulation of the initial value problem for the Localized Induction Equation for a closed vortex filament, and proved its solvability. Their formulation admits integral currents as initial data, enabling them to define the notion of a ``generalized binormal curvature flow'' and a ``weak binormal curvature flow'' (the binormal curvature flow equation is another name for the Localized Induction Equation). They also proved estimates for weak solutions corresponding to perturbed circular filaments. More specifically, they gave an upper bound for the velocity along the symmetry axis.

Kida \cite{Kida1981} constructed solutions of the Localized Induction Equation that travel along an axis without changing their shape. These types of solutions rotate around the axis and exhibit slipping motion along the tangent vector at a constant speed. Garc\'ia and Vega \cite{GarciaVega} and Sampedro and Vega \cite{SampedroVega2026} also constructed solutions of \eqref{eq.LIE} similar to those obtained by Kida \cite{Kida1981}. They first formulated the problem in terms of periodic tangent vectors and constructed non-trivial solutions that are perturbations of the circle.

The papers by Kida \cite{Kida1981}, Garc\'ia and Vega \cite{GarciaVega}, and Sampedro and Vega \cite{SampedroVega2026} are closely related to this paper, and hence we make a more detailed remark after stating our main theorem.

To state the main theorem, we recall the notion of \emph{orbital stability} used in \cite{AikiHigaki2026}. Define the manifold $\Sigma$ as the set obtained from $\bv{x}^R_0$ by rotations around the $\bv{e}_3$-axis and translations:
\begin{equation}\label{def.Sigma}
    \Sigma 
    =
    \{
    Q^\alpha_z \bv{x}^R_0(\cdot) + \bv{\tau}
    ~|~
    \alpha \in \R, \mkern9mu \bv{\tau}\in\R^3
    \}
    \subset H^2(\T)^3, 
\end{equation}
where $Q^\alpha_z$ denotes the rotation matrix around the $\bv{e}_{3}$-axis through angle $\alpha$. We then define the distance between a solution $\bv{x}(s,t)$ of \eqref{eq.LIE} and $\Sigma$ by
\begin{equation}\label{def.dist}
\begin{split}
    \opdist(\bv{x}(t),\Sigma)
    &=
    \inf_{\alpha \in \R, \, \bv{\tau}\in\R^3} 
    \|
    \bv{x}(t)
    - (Q^\alpha_z \bv{x}^R_0 + \bv{\tau})
    \|_{H^2(\T)}. 
\end{split}
\end{equation}
In a previous paper \cite{AikiHigaki2026}, we proved the orbital stability of the circular vortex filament $\bv{x}^R$ with respect to the distance \eqref{def.dist}. More precisely, under suitable smallness assumptions on the initial perturbation, a solution of \eqref{eq.LIE} remains close to the manifold $\Sigma$, that is, to the circular filament up to rotations around the symmetry axis and translations. This notion should be distinguished from \emph{Lyapunov stability} of the representative $\bv{x}^R(t)$ itself, which would require the perturbed solution to stay close to $\bv{x}^R(t)$ for all $t\ge0$ without quotienting by symmetries; see also \cite{Aiki2025} for further discussion.

The goal of this paper is to make this distinction concrete. We construct solutions of \eqref{eq.LIE} that remain uniformly close to $\Sigma$ for all time, yet exhibit secular drift relative to the representative motion $\bv{x}^R(t)$. Notice that tilting the axis of symmetry of $\bv{x}^{R}$ by a small angle achieves this, but we are interested in the existence of such solutions that have a different shape than a circle, and have the same axis of symmetry as $\bv{x}^R$. This is achieved by considering axial screw motions, which we define below.

\begin{definition}\label{def.screw}
A solution $\bv{x}$ of \eqref{eq.LIE} is called an axial screw motion if there exist an angular velocity $\Omega\in\R$, $\bv{y}\in H^4(\T)^3$ with $|\bv{y}_s|=1$ on $\T$, and $c,V\in\R$ such that
\begin{equation}\label{ansatz.screw}
    \bv{x}(s,t)
    =
    Q^{\Omega t}_z \bv{y}(s+ct)
    + Vt \bv{e}_3,
    \quad
    (s,t)\in\T\times[0,\infty).
\end{equation}
\end{definition}
\begin{remark}\label{rem1.def.screw}
For any $\Omega\in\R$, the circular filament can be represented in the screw ansatz as
\[
    \bv{x}^R(s,t)
    =
    Q^{\Omega t}_z \bv{x}^R_0(s-R\Omega t)
    + \frac{t}{R} \bv{e}_3,
    \quad
    (s,t)\in\T\times[0,\infty).
\]
\end{remark}
\begin{remark}
We work in $H^4(\T)^3$, which provides a convenient functional setting for the bifurcation argument below. The unique time-global existence of a solution for the initial-value problem in the Sobolev space $H^4$ is proved in \cite{TaniNishiyama1997}.
\end{remark}
\begin{remark}
From a physical point of view, an axial screw motion is a closed filament whose shape remains unchanged in time: it rotates about the $\bv{e}_{3}$-axis with constant angular velocity $\Omega$, translates along the $\bv{e}_{3}$-axis with constant speed $V$, and undergoes a constant tangential slip with velocity $c$. The last motion only changes the parameterization along the filament and does not affect the shape itself.

We also mention that because we impose the condition $|\bv{y}_{s}|=1$, any axial screw motion can be considered as a perturbation of $\bv{x}^{R}$, since they both have the same length and arc length parameterization. This is a subtle but important point because \eqref{eq.LIE} conserves arc length parameterization and hence, when considering the stability of the circular filament, only perturbations that preserve the arc length parameterization are admissible. 
\end{remark}

Our main theorem below states that, for each $k\in\N$ with $k\ge2$, there exists a family of axial screw motions $\{\bv{x}^{(\lambda)}\}_{\lambda}$ that emerges from the circular vortex filament through Lyapunov-unstable perturbations. In particular, there are infinitely many such families.

\begin{theorem}\label{thm.main}
Let $k\in\N$ satisfy $k\ge2$. Define the critical angular velocity
$\Omega_{k}$ by
\begin{equation}\label{def.Omegak}
    \Omega_k 
    = \frac{\sqrt{k^2-1}}{R^2}.    
\end{equation}
Then there exist $\eps_{k}>0$ and a smooth mapping 
\[
    (-\eps_{k}, \eps_{k}) 
    \ni 
    \lambda 
    \mapsto 
    (\Omega^\lambda, \bv{y}^\lambda, c^\lambda, V^\lambda)
    \in 
    \R \times H^4(\T)^3 \times \R \times \R
\]
with 
\[
    (\Omega^0, \bv{y}^0, c^0, V^0)
    = (\Omega_k, \bv{x}^R_0, -R\Omega_k, 1/R)
\]
such that $(\Omega^\lambda,\bv{y}^\lambda,c^\lambda,V^\lambda)$ defines an axial screw motion $\bv{x}^{(\lambda)}$ by \eqref{ansatz.screw} for each $\lambda\in(-\eps_{k},\eps_{k})$. Moreover, the family $\{\bv{x}^{(\lambda)} \}_{\lambda\in(-\eps_{k},\eps_{k})}$ satisfies the following properties. 
\begin{enumerate}[(1)]
\item\label{item1.thm.main}
There exists $C>0$ such that, for any $\lambda\in(-\eps_{k},\eps_{k})$, 
\[
    \opdist(\bv{x}^{(\lambda)}(t),\Sigma)
    \le
    C |\lambda|,
    \quad
    t\ge 0.
\]

\item\label{item2.thm.main}
For any $\lambda\neq0$, there exist $\gamma,t_0>0$ such that
\[
    |\bv{x}^{(\lambda)}(s,t)-\bv{x}^{R}(s,t)|
    \ge 
    \gamma t, 
    \quad
    (s,t)\in \T\times[t_0,\infty).
\]
This drift is caused by the difference in velocity along the axis of symmetry. In fact, the axial screw motion $\bv{x}^{(\lambda)}(t)$ moves strictly more slowly along the axis of symmetry than $\bv{x}^{R}(t)$. More precisely, when $\lambda \to 0$, 
\[
    V^{\lambda}
    =
    \frac{1}{R}
    - \left(\frac{k^2(k^2-1)}{2R^3}\right) \lambda^2 
    + o(\lambda^2). 
\]
\end{enumerate}
\end{theorem}
We make a few remarks.
\begin{remark}
The exact same statement as Theorem \ref{thm.main} holds with $\Omega_{k}$ replaced by $-\Omega_{k}$. This amounts to considering an axial screw motion which is rotating around the axis of symmetry in the opposite direction. Other than that, the dynamics are the same. The proof for the case $-\Omega_{k}$ is exactly the same as the proof of Theorem \ref{thm.main}, and so we omit the details.
\end{remark}
\begin{remark}\label{rem.thm.main}
The main ingredient to prove Theorem \ref{thm.main} is the local bifurcation theory given by Crandall and Rabinowitz \cite{CrandallRabinowitz1971}. Moreover, from Theorem \ref{thm.main} \eqref{item1.thm.main} and \eqref{item2.thm.main}, we see that all the axial screw motions on the bifurcation branches provide non-trivial perturbations of the circular filament $\bv{x}^R$ which remain close to the manifold $\Sigma$ for all time, but drift away from the representative $\bv{x}^R(t)$ as time evolves. In this sense, these solutions satisfy an orbital-stability-type estimate, but not a Lyapunov-stability-type estimate. 
\end{remark}
\begin{remark}\label{rem.compare}
We compare Theorem \ref{thm.main} with the results in Kida \cite{Kida1981}, Garc\'ia and Vega \cite{GarciaVega}, and Sampedro and Vega \cite{SampedroVega2026}.
\begin{enumerate}
\item 
Kida \cite{Kida1981} constructed solutions of the Localized Induction Equation under an ansatz closely related to our axial screw motion and derived explicit formulas for the profile in terms of elliptic integrals. His construction produces a rich family of filaments moving without change of form. However, his solutions are not formulated as closed constant speed perturbations of a circular filament. In fact, the general solutions in \cite{Kida1981} are not closed, and they yield closed filaments only for special choices of the parameters. By contrast, in the present paper, the relevant class of perturbations must consist of closed filaments with the same total length as the reference circle and parametrized by arc length. Therefore, his results are not directly applicable when considering the stability of the circular filament in this paper.

\end{enumerate}

In both Garc\'ia and Vega \cite{GarciaVega} and Sampedro and Vega \cite{SampedroVega2026}, they utilize the equation for the tangent vector $\bv{T}:=\bv{x}_{s}$. This is obtained by taking the derivative with respect to $s$ of \eqref{eq.LIE} and is given by
\begin{align}
\label{eq.tangent.LIE}
    \bv{T}_{t}=\bv{T}\times \bv{T}_{ss}.
\end{align}
They make an ansatz for $\bv{T}$, which is the $s$-derivative of our axial screw motion, to obtain non-trivial solutions of \eqref{eq.tangent.LIE}.

\begin{enumerate}
    \setcounter{enumi}{1}

\item
In Garc\'ia and Vega \cite{GarciaVega}, they constructed non-trivial solutions by utilizing the local bifurcation theory of Crandall and Rabinowitz \cite{CrandallRabinowitz1971}. 
When the analysis is carried out for the tangent vector, the part of the filament motion that is independent of $s$ must be determined using \eqref{eq.LIE}. Although \cite{GarciaVega} describes how to obtain the position vector from the tangent vector, it does not compare the traveling velocity along the axis of symmetry with that of the reference circle or establish the resulting Lyapunov instability. This comparison is central to the distinction between orbital stability and Lyapunov stability in our theorem.

\item
Approximately 2 months after we published the preprint of our current paper, Sampedro and Vega released their preprint \cite{SampedroVega2026}. Their construction of non-trivial solutions of \eqref{eq.tangent.LIE} bifurcating from the circle overlaps with the local bifurcation analysis in the present paper. They also obtain the corresponding filament motions from the tangent vector and show that these filaments travel along the $\bv{e}_{3}$-axis at a constant speed. However, they do not estimate the difference between this speed and that of a circular filament of the same length. Their main results concern global bifurcation and the structure of the global branches, rather than the stability of the circular filament.

In contrast, Theorem \ref{thm.main} estimates the difference in velocity along the axis of symmetry and shows that the non-trivial axial screw motions move strictly slower than $\bv{x}^{R}(t)$. Because of this difference in velocity, these solutions drift away from the representative motion $\bv{x}^{R}(t)$ at least linearly in time, while remaining uniformly close to $\Sigma$ for all time. These conclusions are not established in \cite{SampedroVega2026}.
\end{enumerate}

In summary, \cite{Kida1981}, \cite{GarciaVega}, and \cite{SampedroVega2026} establish the existence of non-trivial solutions of \eqref{eq.LIE}, but they do not directly address either orbital or Lyapunov stability of circular filaments. This is the point clarified by our main theorem.
\end{remark}

The rest of this paper is organized as follows. Section \ref{sec.screw} collects the basic properties of an axial screw motion in Definition \ref{def.screw}. Section \ref{sec.prf} is devoted to the proof of Theorem \ref{thm.main}.

    \paragraph{Notation}\phantomsection\label{para.notation}

In the rest of this paper, we will use the following notation.
\begin{itemize}
\item
For $m\ge0$, we denote the Sobolev space equipped with the usual inner product by $H^m(\T)$ and the norm in $H^m(\T)$ by $\|\cdot\|_m$.

\item
For $f\in L^1(\T)$ on $\T$, we set $\overbar{f}=(1/2\pi R) \int_{\T} f$. 
\end{itemize}

    \section{Properties of Axial Screw Motions}
    \label{sec.screw}

In this section, we collect the properties of axial screw motions in Definition \ref{def.screw}.

\begin{lemma}\label{lem.screw.dist}
If $\bv{x}$ is an axial screw motion, then
\[
    \opdist(\bv{x}(t),\Sigma)
    = \opdist(\bv{y},\Sigma),
    \quad
    t\ge0.
\]
\end{lemma}

\begin{proof}
Let $\bv{x}$ be given by \eqref{ansatz.screw}. By \eqref{def.dist}, 
\[
    \opdist(\bv{x}(t),\Sigma)
    =
    \inf_{\alpha \in \R, \, \bv{\tau}\in\R^3}
    \|
    Q^{\Omega t}_z \bv{y}(\cdot + ct)
    + Vt \bv{e}_3
    - (Q^\alpha_z \bv{x}^R_0 + \bv{\tau})
    \|_{2}.
\]
Since $\bv{\tau}$ ranges over all $\R^3$, we may absorb $Vt\bv{e}_3$ into $\bv{\tau}$ and obtain
\[
    \opdist(\bv{x}(t),\Sigma)
    =
    \inf_{\alpha \in \R, \, \bv{\tau}\in\R^3}
    \|
    Q^{\Omega t}_z \bv{y}(\cdot + ct)
    - (Q^\alpha_z \bv{x}^R_0 + \bv{\tau})
    \|_{2}.
\]
Using that $Q^{\Omega t}_z$ is orthogonal and $\Sigma$ is invariant under $Q^{\Omega t}_z$,
\[
    \opdist(\bv{x}(t),\Sigma)
    =
    \inf_{\alpha \in \R, \, \bv{\tau}\in\R^3}
    \|
    \bv{y}(\cdot + ct) 
    - (Q^\alpha_z \bv{x}^R_0 + \bv{\tau})
    \|_{2}.
\]
Finally, since $\bv{x}^R_0(s'-ct) = Q^{-ct/R}_z \bv{x}^R_0(s')$, we have for each $\alpha,\bv{\tau}$,
\[
    \|
    \bv{y}(\cdot + ct) 
    - (Q^\alpha_z \bv{x}^R_0 + \bv{\tau})
    \|_{2}
    =
    \|\bv{y}
    - (Q^{\alpha-ct/R}_z \bv{x}^R_0 + \bv{\tau})
    \|_{2}.
\]
Taking the infimum over $\alpha,\bv{\tau}$ implies the assertion.
\end{proof}

\begin{lemma}\label{lem.screw.profile}
Let $\Omega,c,V\in\R$ and $\bv y\in H^4(\T)^3$, and define
\[
    \bv x(s,t)
    =
    Q^{\Omega t}_z \bv y(s+ct)
    + Vt\,\bv e_3.
\]
Then $\bv x$ solves \eqref{eq.LIE} if and only if $\bv y$ satisfies 
\begin{equation}\label{eq.screw.profile}
    \bv{y}_s\times\bv{y}_{ss}
    =
    c \bv{y}_s
    + \Omega \bv{e}_3\times \bv{y}
    + V \bv{e}_3,
    \quad
    s\in\T.
\end{equation}
\end{lemma}

\begin{proof}
Trivially, 
\[
    \bv{x}_s(s,t) = Q^{\Omega t}_z\,\bv{y}_s(s+ct),
    \qquad
    \bv{x}_{ss}(s,t) = Q^{\Omega t}_z\,\bv{y}_{ss}(s+ct).
\]
Moreover,
\[
    \bv{x}_t(s,t)
    =
    \Omega \bv{e}_3
    \times
    (Q^{\Omega t}_z\bv{y}(s+ct))
    + Q^{\Omega t}_z (c \bv{y}_s(s+ct))
    + V \bv{e}_3.
\]
Since $Q^\alpha_z$ preserves exterior products, we have
\[
    (\bv{x}_s\times \bv{x}_{ss})(s,t)
    =
    Q^{\Omega t}_z (\bv{y}_s\times\bv{y}_{ss})(s+ct).
\]
Thus $\bv{x}_t = \bv{x}_s \times \bv{x}_{ss}$ holds if and only if
\[
    Q^{\Omega t}_z (\bv{y}_s \times \bv{y}_{ss})(s+ct)
    =
    \Omega \bv{e}_3\times (Q^{\Omega t}_z\bv{y}(s+ct))
    + Q^{\Omega t}_z (c\bv{y}_s(s+ct))
    + V \bv{e}_3.
\]
Multiplying both sides by $(Q^{\Omega t}_z)^{-1}$ and using $(Q^{\Omega t}_z)^{-1}\bv{e}_3=\bv{e}_3$
yields \eqref{eq.screw.profile}.
\end{proof}

From here on, with a slight abuse of notation, we denote the independent variable of $\bv{y}$ by $s$ and hence, we regard $\bv{y}$ as a map $\bv{y}: \T \ni s \mapsto \bv{y}(s) \in \R^3$.

\begin{proposition}\label{prop.eq.unified}
Let $\bv{y}\in H^4(\T)^3$. Then the following hold.

\begin{enumerate}[(1)]
\item\label{item1.prop.eq.unified}
If $\bv{y}$ satisfies \eqref{eq.screw.profile} and $|\bv{y}_{s}|=1$, then $\bv{y}$ solves 
\begin{equation}\label{eq.unified}
    \bv{y}_{ss}
    =
    - \Omega (\bv{y}_s \cdot \bv{y}) \bv{e}_3 
    + \Omega ( \bv{y}_s \cdot \bv{e}_3) \bv{y} 
    - V \bv{y}_s \times \bv{e}_3,
    \quad
    s\in\T. 
\end{equation}

\item\label{item2.prop.eq.unified}
If $\bv{y}$ satisfies \eqref{eq.unified} and $|\bv{y}_{s}|=1$, then 
\[
    g(s) 
    = 
    - \bv{y}_s(s) 
    \cdot 
    (
    \Omega \bv{e}_3 \times \bv{y}(s) + V \bv{e}_3
    ),
    \quad
    s\in\T, 
\]
is independent of $s$ and $\bv{y}$ solves \eqref{eq.screw.profile} with the constant $c = g(s)$. 
\end{enumerate}
\end{proposition}

\begin{proof}
Suppose $\bv{y}$ satisfies \eqref{eq.screw.profile} and $|\bv{y}_{s}|=1$. Note that
\begin{align*}
    \bv{y}_{s}\cdot \bv{y}_{ss} 
    &= 0, \\
    \bv{y}_{s}\times (\bv{y}_{s}\times \bv{y}_{ss}) 
    &= -\bv{y}_{ss}     
\end{align*}
follow from $|\bv{y}_{s}|^2 =1 $. Hence, taking the exterior product of both sides of equation \eqref{eq.screw.profile} with $\bv{y}_{s}$ from the left yields \eqref{eq.unified}. This proves \eqref{item1.prop.eq.unified} of Proposition \ref{prop.eq.unified}.

Now, suppose $\bv{y}$ satisfies \eqref{eq.unified} and $|\bv{y}_{s}|=1$. Taking the derivative of $g$ with respect to $s$ yields
\begin{align*}
    g_{s}(s) 
    = -\bv{y}_{ss}(s)\cdot 
    (
    \Omega \bv{e}_3 \times \bv{y}(s) + V \bv{e}_3
    ).
\end{align*}
Substituting \eqref{eq.unified} for $\bv{y}_{ss}$ in the above equation shows that $g_{s}=0$ and hence, $g$ is independent of $s$. We denote this constant by $c$, and rewrite the right-hand side of \eqref{eq.unified} to obtain
\begin{align*}
    \bv{y}_{ss}
    &= -\Omega \bv{y}_{s}\times (\bv{e}_{3}\times \bv{y})
    -V\bv{y}_{s}\times \bv{e}_{3}.
\end{align*}
Taking the exterior product of both sides of the above equation with $\bv{y}_{s}$ from the left gives
\begin{align*}
    \bv{y}_{s}\times \bv{y}_{ss}
    &=
    -\Omega \bv{y}_{s}\times \big( \bv{y}_{s}\times (\bv{e}_{3}\times \bv{y})\big)
    -V\bv{y}_{s}\times (\bv{y}_{s}\times \bv{e}_{3}) \\[3pt]
    &= -\Omega  \big( \bv{y}_{s}\cdot(\bv{e}_{3}\times \bv{y})\big)
    \bv{y}_{s} + \Omega \bv{e}_{3}\times \bv{y}
    -V(\bv{e}_{3}\cdot\bv{y}_{s})\bv{y}_{s} + V\bv{e}_{3} \\[3pt]
    &= \big( c + V(\bv{e}_{3}\cdot \bv{y}_{s})\big)\bv{y}_{s}
    + \Omega \bv{e}_{3}\times \bv{y}
    -V(\bv{e}_{3}\cdot\bv{y}_{s})\bv{y}_{s} + V\bv{e}_{3} \\[3pt]
    &= c\bv{y}_{s} + \Omega \bv{e}_{3}\times \bv{y} + V\bv{e}_{3}.
\end{align*}
This shows \eqref{item2.prop.eq.unified} of Proposition \ref{prop.eq.unified}
and finishes the proof.
\end{proof}

Proposition \ref{prop.eq.unified} allows us to work with \eqref{eq.unified} instead of \eqref{eq.screw.profile}. We will therefore formulate the bifurcation problem in terms of \eqref{eq.unified}, recovering the tangential velocity $c$ afterward from Proposition \ref{prop.eq.unified} (\ref{item2.prop.eq.unified}). This reformulation removes $c$ from the unknowns and has an elliptic-type structure that is better suited for the bifurcation argument below.

    \section{Proof of Theorem \ref{thm.main}}
    \label{sec.prf}

In this section, we prove Theorem \ref{thm.main} by applying the bifurcation theorem of Crandall and Rabinowitz \cite{CrandallRabinowitz1971}. More precisely, we first rewrite \eqref{eq.unified} as a perturbation system around the circular profile $\bv{x}^R_0$. We then use the implicit function theorem to eliminate the auxiliary variables and obtain a reduced equation for the essential variables. After analyzing the linearized operator at the trivial branch, we construct a smooth local branch of non-trivial axial screw motions bifurcating from the circular filament $\bv{x}^R$.

    \subsection{The Perturbation System}
    \label{sec.pert.sys}

Substituting the ansatz 
\begin{equation}\label{ansatz.z-delV}
    \bv{y} = \bv{x}^R_0 + \bv{z},
    \qquad
    V = \frac{1}{R}+ \del V 
\end{equation}
into the unified equation \eqref{eq.unified}, we have 
\[
\begin{split}
    \bv{0}
    &= 
    -\bv{y}_{ss}
    - \Omega (\bv{y}_s \cdot \bv{y}) \bv{e}_3 
    + \Omega ( \bv{y}_s \cdot \bv{e}_3) \bv{y} 
    - V \bv{y}_s \times \bv{e}_3 \\
    &=
    - \bv{x}^R_{0ss} - \bv{z}_{ss} 
    - \Omega 
    \big(
    (\bv{x}^R_{0s} + \bv{z}_s) \cdot (\bv{x}^R_0 + \bv{z}) 
    \big) \bv{e}_3 \\
    &\quad 
    + \Omega 
    \big(
    (\bv{x}^R_{0s} + \bv{z}_s) \cdot \bv{e}_3
    \big) 
    (\bv{x}^R_0 + \bv{z}) 
    - \Big(\frac{1}{R}+ \del V\Big) (\bv{z}_s + \bv{x}^R_{0s}) 
    \times
    \bv{e}_3.  
\end{split}
\]
Using 
\[
    \bv{x}^R_{0ss} = -\frac{1}{R^2} \bv{x}^R_0,
    \qquad
    \bv{x}^R_{0s} \cdot \bv{x}^R_0 = 0,
    \qquad
    \bv{x}^R_{0s} \cdot \bv{e}_3 = 0,
    \qquad
    \bv{x}^R_{0s} \times \bv{e}_3 = \frac{1}{R} \bv{x}^R_0, 
\]
and rewriting the right-hand side, we obtain the equation for $(\del V,\bv{z})$ as 
\begin{equation}\label{eq.z-delV}
\begin{split}
    &-\bv{z}_{ss} 
    - \Omega 
    \Big(
    \bv{z}_s\cdot \bv{z}
    + \frac{\dd}{\dd s}(\bv{z}\cdot \bv{x}^R_0)
    \Big) 
    \bv{e}_3 \\[3pt]
    &
    + \Omega 
    (\bv{z}_s \cdot \bv{e}_3) 
    (\bv{z} + \bv{x}^R_0) 
    - \Big(\frac{1}{R} + \del V\Big)
    \bv{z}_s \times \bv{e}_3 
    - \del V
    \bv{x}^R_{0s} \times \bv{e}_3
    = 0,
    \quad
    s\in\T. 
\end{split}
\end{equation}

Define the unit tangent, normal, and binormal vectors by, respectively, 
\begin{equation}\label{def.frame}
    \bv{t}(s) 
    = \bv{x}^R_{0s}(s), 
    \qquad
    \bv{n}(s) 
    = -\frac{1}{R} \bv{x}^R_0(s), 
    \qquad
    \bv{b} 
    = \bv{e}_3.
\end{equation}
Then the Frenet-Serret formulas are given by
\begin{equation}\label{eq.FS}
    \bv{t}_s 
    = \frac{1}{R} \bv{n}, 
    \qquad
    \bv{n}_s 
    = -\frac{1}{R} \bv{t}, 
    \qquad
    \bv{b}_s 
    = \bv{0}.
\end{equation}
Expand $\bv{z}$ with respect to $\{\bv{t},\bv{n},\bv{b}\}$ as 
\begin{equation}\label{eq.z}
    \bv{z}
    = u \bv{t} 
    + v \bv{n} 
    + w \bv{b}.
\end{equation}
Observe from \eqref{eq.FS} that
\[
\begin{split}
    \bv{z}_s 
    &= 
    \Big(
    u_s - \frac{1}{R} v
    \Big)
    \bv{t} 
    + \Big(
    v_s + \frac{1}{R} u
    \Big)
    \bv{n} 
    + w_s \bv{b}, \\[3pt]
    \bv{z}_{ss} 
    &= 
    \Big(
    u_{ss} - \frac{2}{R} v_s - \frac{1}{R^2} u 
    \Big)
    \bv{t} 
    + \Big(
    v_{ss} + \frac{2}{R} u_s - \frac{1}{R^2} v 
    \Big)
    \bv{n} 
    + w_{ss}\bv{b}. 
\end{split}
\]
Then, using
\[
    \bv{z}\cdot \bv{x}^R_0 
    = \bv{z}\cdot (-R\bv{n})
    = -R v,
    \qquad
    \bv{z}_s\cdot \bv{e}_3
    = \bv{z}_s\cdot \bv{b} 
    = w_s, 
\]
and the formula following from $\bv{t} \times \bv{b} = -\bv{n}$ and $\bv{n} \times \bv{b} = \bv{t}$: 
\[
    \bv{z}_s \times \bv{e}_3
    = \bv{z}_s \times \bv{b}
    = -\Big(
    u_s - \frac{1}{R} v
    \Big)
    \bv{n} 
    + \Big(
    v_s + \frac{1}{R} u
    \Big)
    \bv{t}, 
\]
we can express the equation \eqref{eq.z-delV} as 
\begin{equation}\label{eq.TNB}
\begin{split}
    \mathcal{T}(s) \bv{t}(s)
    + \mathcal{N}(s) \bv{n}(s)
    + \mathcal{B}(s) \bv{b}
    = \bv{0}.
\end{split}
\end{equation}
Here 
\begin{equation}\label{def.TNB}
\begin{split}
    \mathcal{T}
    &:=
    -u_{ss} + \frac{1}{R}v_{s} + \Omega uw_{s} 
    - \delta V \Big(v_s + \frac{1}{R} u\Big), \\[3pt]
    \mathcal{N}
    &:=
    -v_{ss} - \frac{1}{R} u_s - R\Omega w_s + \Omega vw_s 
    + \delta V \Big(1 + u_s - \frac{1}{R} v\Big), \\[3pt]
    \mathcal{B}
    &:=
    -w_{ss} + R\Omega v_s - \Omega uu_s - \Omega vv_s.
\end{split}
\end{equation}
Moreover, by
\[
    \bv{y}_{s}
    =
    \bv{x}^R_{0s}+ \bv{z}_{s}
    =
    \Big(
    1 + u_{s} - \frac{1}{R} v
    \Big)
    \bv{t} 
    + \Big(
    v_{s} + \frac{1}{R} u
    \Big)
    \bv{n} 
    + w_{s}
    \bv{b},
\]
the condition $|\bv{y}_s|^2 =1$ can be expressed as $\mathcal{C}=0$, where
\begin{equation}\label{def.C}
    \mathcal{C} 
    := 
    u_s 
    - \frac{1}{R} v 
    + \mathcal{Q},
    \qquad
    \mathcal{Q}
    :=
    \frac12 
    \bigg( 
    \Big(u_{s} - \frac{1}{R} v\Big)^2 
    + \Big(v_{s}+ \frac{1}{R} u\Big)^2 
    + w_{s}^2 
    \bigg).
\end{equation}
Observe that
\[
    |\bv y_s|^2
    =
    \Big(
    1 + u_s - \frac{1}{R} v
    \Big)^2
    + \Big(
    v_s + \frac{1}{R} u
    \Big)^2
    + w_s^2
    =
    1 + 2\mathcal{C}.
\]
Hence, taking the inner product of \eqref{eq.unified} with $\bv y_s$, we obtain
\begin{equation}\label{id.TNBC}
    \Big(
    1 + u_s - \frac{1}{R} v
    \Big) \mathcal{T}
    + \Big(
    v_s + \frac{1}{R} u
    \Big) \mathcal{N}
    + w_s \mathcal B
    =
    -\frac12 \partial_s |\bv y_s|^2
    =
    -\partial_s \mathcal{C}.
\end{equation}

Finally, the perturbation system we will address is given by 
\begin{equation}\label{eq.NBC}
    \mathcal{N}(s)
    = \mathcal{B}(s)
    = \mathcal{C}(s)
    = 0,
    \quad
    s\in\T,
\end{equation}
where $\mathcal{N},\mathcal{B}$ are defined in \eqref{def.TNB} and $\mathcal{C}$ is in \eqref{def.C}. We emphasize that the tangential equation $\mathcal{T}=0$ is not imposed directly in \eqref{eq.NBC}. Instead, for solutions sufficiently close to the trivial state, we recover $\mathcal{T}=0$ a posteriori from \eqref{id.TNBC}, once $\mathcal{N}=\mathcal B=\mathcal{C}=0$ and a lower bound such as $1+u_s-v/R\ge 1/2$ are obtained; see Section \ref{sec.end} for details.

    \subsection{Reduction of the System}
    \label{sec.redn.sys}

To apply the local bifurcation theory, we must identify the essential degrees of freedom that govern the nonlinear system \eqref{eq.NBC}. Let us first decompose $v$ as
\[
    v = v_0 + v^\perp,
    \qquad
    v_0 := \overbar{v} = \frac{1}{2\pi R} \int_{\T} v, 
    \qquad
    v^\perp := v - v_0.
\]
Then, by varying the angular velocity $\Omega$, we look for critical values where the axial screw motions emerge. In this analysis, the roles of the variables 
\[
    \del V, 
    \quad
    u, 
    \quad
    v_0, 
    \quad
    v^\perp, 
    \quad
    w
\]
are understood geometrically as follows:

\begin{itemize}
\item
The tangential component $u$ corresponds to the reparameterization of the curve, which should be completely determined by the arc length constraint $\mathcal{C}=0$.

\item
The constant $v_0$ corresponds to a trivial expansion of the reference circle, and $\delta V$ to a shift in the translational velocity. These constants are expected to balance the averages in the normal constraint $\mathcal{N}=0$ and the arc length constraint $\mathcal{C}=0$.

\item
The normal component $v^\perp$ and the binormal component $w$ are the genuine profile of the bifurcating motions, and will therefore serve as our independent variables.
\end{itemize}

With this geometric insight, we will eliminate the variables $\del V,u,v_0$ by regarding them as functions of the independent variables $\Omega,v^\perp,w$. Define the function spaces
\[
\begin{aligned}
    X_{\rm odd} &= \{f\in H^4(\T)~|~\mbox{$f$ is odd}\}, & \qquad
    Y_{\rm odd} &= \{f\in H^2(\T)~|~\mbox{$f$ is odd}\}, \\
    X_{\rm even} &= \{f\in H^4(\T)~|~\mbox{$f$ is even}\}, & \qquad
    Y_{\rm even} &= \{f\in H^2(\T)~|~\mbox{$f$ is even}\}, \\
    \dot{X}_{\rm even} &= \{f\in X_{\rm even}~|~\overbar{f}=0\}, & \qquad
    \dot{Y}_{\rm even} &= \{f\in Y_{\rm even}~|~\overbar{f}=0\}, 
\end{aligned}
\]
and $Z_{\rm odd},Z_{\rm even},\dot{Z}_{\rm even}$
similarly with the $H^3(\T)$ topology. Since $\partial_s: X_{\rm odd}\to \dot Z_{\rm even}$ is an isomorphism, we denote its inverse by $\partial_s^{-1}:\dot Z_{\rm even}\to X_{\rm odd}$. Moreover, we define 
\begin{equation}\label{def.F1}
\begin{split}
    \mathcal{F} 
    = 
    (\mathcal{F}_1, \mathcal{F}_2, \mathcal{F}_3): 
    \R \times \dot{X}_{\rm even} \times X_{\rm odd}
    \times \R \times X_{\rm odd} \times \R
    \to
    \R \times \dot{Z}_{\rm even} \times \R
\end{split}
\end{equation}
by
\begin{equation}\label{def.F2}
\begin{split}
    \mathcal{F}_1
    &=
    \mathcal{F}_1(\Omega,v^\perp,w,\del V,u,v_0) 
    = 
    \overbar{\mathcal{N}}, \\
    \mathcal{F}_2
    &=
    \mathcal{F}_2(\Omega,v^\perp,w,\del V,u,v_0)
    = 
    \mathcal{C} - \overbar{\mathcal{C}}, \\
    \mathcal{F}_3
    &=
    \mathcal{F}_3(\Omega,v^\perp,w,\del V,u,v_0) 
    = 
    \overbar{\mathcal{C}}. 
\end{split}
\end{equation}
More explicitly, using $\mathcal{Q}$ in \eqref{def.C}, 
\begin{equation}\label{def.F}
\begin{split}
    \mathcal{F}_1
    &= 
    \Omega \overbar{v^\perp w_s}
    + \delta V \Big(1 - \frac{1}{R} v_0\Big), \\[3pt]
    \mathcal{F}_2
    &= 
    u_s 
    - \frac{1}{R} v^\perp 
    + \mathcal{Q} - \overbar{\mathcal{Q}}, \\[3pt]
    \mathcal{F}_3
    &= 
    - \frac{1}{R} v_0
    + \overbar{\mathcal{Q}}. 
\end{split}
\end{equation}
Note that $\mathcal{F}(\Omega,0,0,0,0,0)=0$ for any $\Omega\in\R$; see also Remark \ref{rem1.def.screw}.

Using the function spaces above, we restrict the variables $v^\perp$ and $w$ to the subspaces $\dot X_{\rm even}$ and $X_{\rm odd}$, respectively, to factor out the phase-shift of Fourier modes and obtain a one-dimensional kernel for the Crandall–Rabinowitz theorem; see Lemma \ref{lem.CRT}.

\begin{proposition}\label{prop.IFT}
For any $\Omega\in\R$, there exist open sets $U,W$: 
\[
\begin{split}
    (\Omega,0,0) 
    &\in
    U
    \subset
    \R \times \dot{X}_{\rm even} \times X_{\rm odd}, \\
    (0,0,0) 
    &\in
    W
    \subset
    \R \times X_{\rm odd} \times \R, 
\end{split}
\]
and a unique $\varphi:U\to W$ 
such that $\varphi:(\Omega, v^\perp, w) \mapsto (\del V,u,v_{0})$ and, for $(\Omega, v^\perp, w)\in U$,
\[
    \mathcal{F}(\Omega,v^\perp,w,\varphi(\Omega,v^\perp,w))
    = (0,0,0). 
\]
Moreover, $\varphi$ is smooth, and its Fr\'{e}chet derivative with respect to $(v^\perp,w)$ at $(\Omega,0,0)$ is 
\begin{equation}\label{eq.prop.IFT}
    D_{(v^\perp,w)}\,\varphi(\Omega,0,0)
    \begin{pmatrix}
    h_{v^\perp} \\
    h_w
    \end{pmatrix}
    = 
    \begin{pmatrix}
    0 \\
    (1/R)\partial_s^{-1}h_{v^\perp}\\
    0
    \end{pmatrix}.
\end{equation}
\end{proposition}

\begin{proof}
We apply the implicit function theorem in Banach spaces to $\mathcal{F}$ at the point $\Lambda_0:=(\Omega,0,0,0,0,0)$; see for example \cite[Chapter 4]{Deimling1985_Book}. By the definition of $\mathcal{F}=(\mathcal{F}_1,\mathcal{F}_2,\mathcal{F}_3)$ in \eqref{def.F}, we calculate the Fr\'{e}chet derivative as 
\[
    D_{(\delta V,u,v_0)}\,\mathcal{F}(\Lambda_0)
    =
    \begin{pmatrix}
    1 & 0 & 0 \\
    0 & \partial_s & 0 \\
    0 & 0 & -1/R
    \end{pmatrix} 
    : \R \times X_{\rm odd} \times \R 
    \to 
    \R \times \dot{Z}_{\rm even} \times \R. 
\]
Then $D_{(\delta V,u,v_0)}\,\mathcal{F}(\Lambda_0)$ is an isomorphism since $\partial_s:X_{\rm odd} \to \dot{Z}_{\rm even}$ is an isomorphism. Hence the implicit function theorem yields the existence of the open sets $U$, $W$,
and the unique smooth map $\varphi:U\to W$ satisfying
\[
    \mathcal{F}(\Omega,v^\perp,w,\varphi(\Omega,v^\perp,w))
    =
    (0,0,0),
    \quad
    (\Omega,v^\perp,w)\in U. 
\]

To compute the Fr\'{e}chet derivative of $\varphi$, we differentiate the identity 
\[
    \mathcal{F}
    (\Omega,v^\perp,w,\varphi(\Omega,v^\perp,w)) 
    = (0,0,0) 
\]
with respect to $(v^\perp,w)$ to obtain
\[
    D_{(v^\perp,w)}\,\mathcal{F}(\Lambda_0)
    + D_{(\delta V,u,v_0)}\,\mathcal{F}(\Lambda_0) 
    \circ D_{(v^\perp,w)}\,\varphi(\Omega,0,0)
    = 0, 
\]
and equivalently, 
\[
    D_{(v^\perp,w)}\,\varphi(\Omega,0,0)
    = 
    -\big(
    D_{(\delta V,u,v_0)}\,\mathcal{F}(\Lambda_0)
    \big)^{-1} 
    \circ
    D_{(v^\perp,w)}\,\mathcal{F}(\Lambda_0).
\]
Hence the assertion follows from 
\[
    \big(
    D_{(\delta V,u,v_0)}\,\mathcal{F}(\Lambda_0)
    \big)^{-1}
    =
    \begin{pmatrix}
    1 & 0 & 0 \\
    0 & \partial_s^{-1} & 0 \\
    0 & 0 & -R
    \end{pmatrix}
    : \R \times \dot{Z}_{\rm even} \times \R
    \to
    \R \times X_{\rm odd} \times \R
\]
and 
\[
    D_{(v^\perp,w)}\,\mathcal{F}(\Lambda_0)
    \begin{pmatrix}
    h_{v^\perp} \\
    h_w
    \end{pmatrix} 
    = 
    \begin{pmatrix} 
    0 \\ 
    -(1/R)h_{v^\perp} \\ 
    0 
    \end{pmatrix}. 
\]
This completes the proof. 
\end{proof}

Using $U,\varphi$ in Proposition \ref{prop.IFT}, we define the reduced operator
\begin{equation}\label{def.G1}
    \mathcal{G} = (\mathcal{G}_1, \mathcal{G}_2):
    U
    \to 
    \dot{Y}_{\rm even} \times Y_{\rm odd}
\end{equation}
by
\begin{equation}\label{def.G2}
\begin{split}
    \mathcal{G}_1(\Omega,v^\perp,w) 
    &= 
    (\mathcal{N} - \overbar{\mathcal{N}}) 
    (\Omega,v^\perp,w,\varphi(\Omega,v^\perp,w)), \\
    \mathcal{G}_2(\Omega,v^\perp,w) 
    &= 
    \mathcal{B}(\Omega,v^\perp,w,\varphi(\Omega,v^\perp,w)).   
\end{split}
\end{equation}

    \subsection{The Linearized Operator}
    \label{sec.lin.op}

For fixed $\Omega$, we define the linearized operator $\mathcal{L}$ on $\dot{Y}_{\rm even} \times Y_{\rm odd}$ by 
\[
    \mathcal{L}
    := D_{(v^\perp,w)}\,\mathcal{G}(\Omega,0,0)\big|_{\dot{X}_{\rm even} \times X_{\rm odd}}, 
    \qquad
    D(\mathcal{L})
    = \dot{X}_{\rm even} \times X_{\rm odd}.
\]
We equip $\dot Y_{\rm even}\times Y_{\rm odd}$ with the standard $H^2$ inner product
\[
    \bigg\langle
    \begin{pmatrix}
    f_1 \\
    f_2
    \end{pmatrix}
    , 
    \begin{pmatrix}
    g_1 \\
    g_2
    \end{pmatrix}
    \bigg\rangle_{H^2}
    :=
    \sum_{m=1}^2\sum_{j=0}^2
    \int_{\T}
    \partial_s^j f_m\,\partial_s^j g_m
    \dd s.
\]
We will also use the notation $\mathcal{L}_{\Omega}$, instead of $\mathcal{L}$, to denote the linearized operator when we want to emphasize the value of $\Omega$. The explicit form of $\mathcal{L}$ is obtained as follows: we expand $\mathcal{N}$ and $\mathcal{B}$ in \eqref{def.TNB} to the first order. By \eqref{eq.prop.IFT}, we have $u_s=(1/R)v^\perp$ and $\delta V=v_0=0$ to the first order. Substituting these relations into $\mathcal{N}$ and $\mathcal{B}$, and noting that $\overbar{v^\perp}=0$, we obtain 
\begin{equation}\label{def.L}
    \mathcal{L}
    \begin{pmatrix}
    v^\perp \\
    w
    \end{pmatrix}
    =
    \begin{pmatrix}
    -v^\perp_{ss} - (1/R^2) v^\perp - R\Omega w_s \\[3pt]
    -w_{ss} + R\Omega v^\perp_s
    \end{pmatrix}.
\end{equation}

Let us compute the Fourier representation of $\mathcal{L}=\mathcal{L}_\Omega$. If
\[
    v^\perp(s) 
    = \sum_{l\ge1} 
    \hat{v}^\perp_l \cos\Big(\frac{ls}{R}\Big),
    \qquad
    w(s) 
    = \sum_{l\ge1}
    \hat{w}_l \sin\Big(\frac{ls}{R}\Big),
\]
then
\[
    \mathcal{L}_\Omega
    \begin{pmatrix}
    v^\perp \\
    w
    \end{pmatrix}
    =
    \sum_{l\ge1}
    \begin{pmatrix}
    \big(
    ((l^2-1)/R^2)\hat{v}^\perp_l - l\Omega \hat{w}_l
    \big)
    \cos(ls/R) \\[3pt]
    \big(
    (l^2/R^2)\hat{w}_l - l\Omega \hat{v}^\perp_l
    \big)
    \sin(ls/R)
    \end{pmatrix}.
\]
Thus, on each Fourier mode, $\mathcal{L}_\Omega$ acts through the matrix
\[
    M_l(\Omega)
    :=
    \begin{pmatrix}
    (l^2-1)/R^2 & -l\Omega \\[3pt]
    -l\Omega & l^2/R^2
    \end{pmatrix},
    \quad
    l \ge 1.
\]
Then it is routine to check that $\mathcal{L}$ is a self-adjoint operator on $\dot{Y}_{\rm even} \times Y_{\rm odd}$.

To find the non-trivial solutions of $\mathcal{L}\,{}^{t}(v^\perp\,w)=0$, we fix $k\ge1$ and consider 
\[
    \begin{pmatrix}
    v^\perp \\
    w
    \end{pmatrix}
    =
    \begin{pmatrix}
    \hat{v}^\perp_k \cos(k s/R) \\[3pt]
    \hat{w}_k \sin(k s/R)
    \end{pmatrix}. 
\]
After the substitution, we obtain 
\[
    M_{k}(\Omega)
    \begin{pmatrix}
    \hat{v}^\perp_k \\
    \hat{w}_k
    \end{pmatrix}
    =
    \begin{pmatrix}
    (k^2-1)/R^2 & -k\Omega \\[3pt]
    -k\Omega & k^2/R^2 \\
    \end{pmatrix}
    \begin{pmatrix}
    \hat{v}^\perp_k \\
    \hat{w}_k
    \end{pmatrix}
    =
    \begin{pmatrix}
    0 \\
    0
    \end{pmatrix}.
\]
A direct computation shows that
\[
    \det M_{k}(\Omega)
    = 
    k^2
    \Big(
    \frac{k^2-1}{R^4} - \Omega^2
    \Big), 
\]
from which we derive the bifurcation condition 
\[
    \Omega 
    = \pm \frac{\sqrt{k^2-1}}{R^2}.
\]
In the following, we focus on the non-negative branch $\sqrt{k^2-1}/R^2$,
which we denote by $\Omega_k$ consistently with \eqref{def.Omegak}.
 Another direct computation shows that, when $\Omega=\Omega_k$, 
\[
    \det M_l(\Omega_k)
    = 
    \frac{l^2(l^2-k^2)}{R^4}, 
    \quad
    l \ge 1,
\]
and that the null space $\opN(M_k(\Omega_k))$ is given by
\[
    \opN(M_k(\Omega_k))
    =
    \opspan
    \bigg\{
    \begin{pmatrix}
    k \\[3pt]
    \sqrt{k^2-1}
    \end{pmatrix} 
    \bigg\}. 
\]
Hence the non-trivial solutions when $\Omega=\Omega_{k}$ are the elements belonging to 
\begin{equation}\label{def.Phi}
    \opN(\mathcal{L}_{\Omega_k})
    = 
    \opspan\{\Phi_k\},
    \qquad
    \Phi_k
    = 
    \begin{pmatrix}
    (\Phi_k)_1 \\[3pt]
    (\Phi_k)_2
    \end{pmatrix} 
    = 
    \begin{pmatrix}
    k \cos(k s/R) \\[3pt]
    \sqrt{k^2-1} \sin(k s/R)
    \end{pmatrix}. 
\end{equation}

    \subsection{End of the Proof}
    \label{sec.end}

To apply the bifurcation theory by Crandall and Rabinowitz, we need the following lemma.

\begin{lemma}\label{lem.CRT}
The following hold.

\begin{enumerate}[(1)]
\item\label{item1.lem.CRT}
$\mathcal{L}_{\Omega_k}$ is a Fredholm operator and there exists $\Phi_k$ such that $\opN(\mathcal{L}_{\Omega_k}) = \opspan \{\Phi_k\}$.

\item\label{item2.lem.CRT}
$\opcodim\opR(\mathcal{L}_{\Omega_k})=1$, where
$\opR(\mathcal{L}_{\Omega_k})$ is the range of $\mathcal{L}_{\Omega_k}$.

\item\label{item3.lem.CRT}
Let $k\in\N$ satisfy $ k\ge2$. Then 
$
\mathcal{L}'_{\Omega_k} \Phi_k 
\notin \opR(\mathcal{L}_{\Omega_k})
$ where 
$
\mathcal{L}'_{\Omega_k} 
:= \partial_\Omega \mathcal{L}_\Omega|_{\Omega=\Omega_k}
$. 
\end{enumerate}
\end{lemma}

\begin{proof}
For (\ref{item1.lem.CRT}), the first claim follows from the ellipticity of $\mathcal{L}_{\Omega_k}$, and the second one from the results in Section \ref{sec.lin.op}. For (\ref{item2.lem.CRT}), since $\mathcal{L}_{\Omega_k}$ is self-adjoint on the Hilbert space $\dot Y_{\rm even}\times Y_{\rm odd}$ endowed with the $H^2$ inner product and is a Fredholm operator, its range is closed and
\begin{equation}\label{eq.lem.CRT}
    \opR(\mathcal{L}_{\Omega_k})
    = \opN(\mathcal{L}_{\Omega_k})^\perp.
\end{equation}
Thus the claim follows from \eqref{def.Phi}. For (\ref{item3.lem.CRT}), thanks to \eqref{eq.lem.CRT}, the claim is equivalent to
\[
    \langle
    \mathcal{L}'_{\Omega_k} \Phi_k, \Phi_k
    \rangle_{H^2}
    \neq
    0.
\]
In view of the formula derived from \eqref{def.L},
\[
    \mathcal{L}'_{\Omega_k}
    \begin{pmatrix}
    v^\perp \\
    w
    \end{pmatrix}
    =
    \begin{pmatrix}
    - R w_s \\
    R v^\perp_s
    \end{pmatrix}
\]
and we have
\[
    \mathcal L'_{\Omega_k}\Phi_k
    =
    \begin{pmatrix}
    -k\sqrt{k^2-1}\cos(ks/R) \\[3pt]
    -k^2\sin(ks/R)
    \end{pmatrix}.
\]
Therefore,
\[
\begin{split}
    &\langle
    \mathcal{L}'_{\Omega_k} \Phi_k, \Phi_k
    \rangle_{H^2} \\
    &=
    \sum_{j=0}^2
    \int_{0}^{2\pi R}
    \partial_s^j
    \bigg(
    -k\sqrt{k^2-1}\cos\Big(\frac{ks}{R}\Big)
    \bigg)
    \partial_s^j
    \bigg(
    k\cos\Big(\frac{ks}{R}\Big)
    \bigg)
    \dd s \\
    &\quad
    +
    \sum_{j=0}^2
    \int_{0}^{2\pi R}
    \partial_s^j
    \bigg(
    -k^2\sin\Big(\frac{ks}{R}\Big)
    \bigg)
    \partial_s^j
    \bigg(
    \sqrt{k^2-1}\sin\Big(\frac{ks}{R}\Big)
    \bigg)
    \dd s \\
    &=
    -2\pi R\,k^2\sqrt{k^2-1}
    \Big(
    1 + \frac{k^2}{R^2} + \frac{k^4}{R^4}
    \Big).
\end{split}
\]
Then the right-hand side is nonzero due to $k\ge2$. Thus the claim follows.
\end{proof}

Set
\[
    G(\mu,v^\perp,w)
    = 
    \mathcal{G}(\Omega_k+\mu,v^\perp,w),
\]
where $\mathcal{G}$ is defined in \eqref{def.G2}. Then we see from Lemma \ref{lem.CRT} that $G$ satisfies the hypotheses of \cite[Theorem 1.7]{CrandallRabinowitz1971}. Thus there exist $\eps_{k}>0$ and a continuously differentiable curve 
\[
    (-\eps_{k}, \eps_{k}) 
    \ni 
    \lambda 
    \mapsto 
    (\Omega^\lambda, v^{\perp,\lambda}, w^\lambda) 
    \in 
    \R \times \dot{X}_{\rm even} \times X_{\rm odd}
\]
with 
$
(\Omega^0, v^{\perp,0}, w^0) = (\Omega_k, 0, 0)
$ 
such that 
$
\mathcal{G}(\Omega^\lambda, v^{\perp,\lambda}, w^\lambda) = (0,0)
$
for all $\lambda \in (-\eps_{k}, \eps_{k})$. Since $\mathcal{G}$ is smooth, \cite[Theorem 1.18]{CrandallRabinowitz1971} yields that this curve is in fact smooth in $\lambda$.

The variables $\Omega^\lambda,v^{\perp,\lambda},w^\lambda$ are expressed using $\Phi_k$ defined in \eqref{def.Phi} as follows:
\begin{equation}\label{est1.prf.thm.main}
    \Omega^\lambda
    = 
    \Omega_k 
    + \tilde{\Omega}^{\lambda},
    \qquad 
    v^{\perp,\lambda}
    = 
    \lambda (\Phi_k)_1 
    + \tilde{v}^{\perp,\lambda}, 
    \qquad
    w^\lambda
    = 
    \lambda (\Phi_k)_2 
    + \tilde{w}^{\lambda}. 
\end{equation}
Here the remainders $\tilde{\Omega}^{\lambda},\tilde{v}^{\perp,\lambda},\tilde{w}^{\lambda}$ satisfy, when $\lambda \to 0$, 
\begin{equation}\label{est2.prf.thm.main}
    |\tilde{\Omega}^\lambda| 
    = O(|\lambda|),
    \qquad
    \|\tilde{v}^{\perp,\lambda}\|_{4} 
    = o(|\lambda|),
    \qquad
    \|\tilde{w}^{\lambda}\|_{4} 
    = o(|\lambda|).
\end{equation}
Recall from Proposition \ref{prop.IFT} that the variables $\delta V^\lambda,u^\lambda,v^{\lambda}_0$ are determined by $\varphi$ as
\[
    (\delta V^\lambda, u^\lambda, v^{\lambda}_0)
    =
    \varphi(\Omega^\lambda, v^{\perp,\lambda}, w^\lambda)
    \in 
    \R \times X_{\rm odd} \times \R.
\]
The smoothness of $\varphi$ combined with \eqref{eq.prop.IFT} and \eqref{est1.prf.thm.main} shows that, when $\lambda \to 0$, 
\begin{equation}\label{est3.prf.thm.main}
    u^\lambda(s)
    = 
    \lambda \sin\Big(\frac{ks}{R}\Big) 
    + \tilde{u}^\lambda(s), 
    \qquad
    \delta V^\lambda 
    = 
    \del \tilde{V}^\lambda, 
    \qquad
    v_0^\lambda 
    = 
    \tilde{v}_0^\lambda, 
\end{equation}
where the remainders 
$\tilde{u}^{\lambda},\del \tilde{V}^\lambda,\tilde{v}_0^\lambda$ 
satisfy 
\begin{equation}\label{est4.prf.thm.main}
    \|\tilde{u}^\lambda\|_{4} 
    = o(|\lambda|),    
    \qquad
    |\del \tilde{V}^\lambda| 
    = O(\lambda^2),
    \qquad
    |\tilde{v}_0^\lambda| 
    = O(\lambda^2).
\end{equation}
By construction, the obtained tuple 
\[
    \Omega^\lambda,
    \quad
    \del V^\lambda,
    \quad
    u^\lambda,
    \quad
    v^\lambda:= v^\lambda_0 + v^{\perp,\lambda}, 
    \quad
    w^\lambda   
\]
solves the nonlinear system \eqref{eq.NBC}, that is, $\mathcal{N}=\mathcal B=\mathcal{C}=0$. It is straightforward to check that this tuple satisfies $\mathcal{T}=0$ by using \eqref{id.TNBC} and the Sobolev embedding $H^2(\T)\hookrightarrow C^1(\T)$ and by taking $\eps_{k}$ smaller if necessary so that $1+u^{\lambda}_s-v^{\lambda}/R\ge 1/2$ holds.

Now we construct $\bv{y}^\lambda \in H^4(\T)^3$ by substituting the variables back into \eqref{ansatz.z-delV} and \eqref{eq.z}: 
\begin{equation}\label{est5.prf.thm.main}
    \bv{y}^\lambda
    = 
    \bv{x}^R_0
    + u^\lambda \bv{t}
    + (v^{\lambda}_0 + v^{\perp,\lambda}) \bv{n}
    + w^\lambda \bv{b}. 
\end{equation}
Then the corresponding axial screw motion $\bv{x}^{(\lambda)}$ is obtained from \eqref{ansatz.screw} as 
\begin{equation}\label{est6.prf.thm.main}
    \bv{x}^{(\lambda)}(s,t)
    =
    Q^{\Omega^\lambda t}_z \bv{y}^\lambda(s + c^\lambda t)
    + \Big(\frac1R + \del V^\lambda \Big) t \bv{e}_3,
    \quad
    (s,t)\in\T\times[0,\infty), 
\end{equation}
with the constant $c^\lambda$ defined by 
\begin{equation}\label{est7.prf.thm.main}
    c^\lambda
    = 
    - \bv{y}^\lambda_s 
    \cdot 
    \bigg(
    \Omega^\lambda \bv{e}_3 \times \bv{y}^\lambda 
    + \Big(\frac1R + \del V^\lambda \Big) \bv{e}_3
    \bigg). 
\end{equation}
Recall that $c^\lambda$ actually does not depend on $s\in\T$ by Proposition \ref{prop.eq.unified} (\ref{item2.prop.eq.unified}). Moreover, since 
\[
    (\delta V^0,u^0,v_0^0,v^{\perp,0},w^0)
    = (0,0,0,0,0), 
\]
we see from \eqref{est5.prf.thm.main} that $\bv{y}^0=\bv{x}^R_0$. In addition, \eqref{est7.prf.thm.main} gives 
\[
    c^0
    =
    - \bv{x}^R_{0s}\cdot
    \Big(
    \Omega_k \bv{e}_3\times \bv{x}^R_0
    + \frac1R \bv{e}_3
    \Big)
    =
    -R\Omega_k, 
\]
which leads to $\bv{x}^{(0)}=\bv{x}^R$.

In the following, we show the properties (\ref{item1.thm.main}) and (\ref{item2.thm.main}). By Lemma \ref{lem.screw.dist} and $\bv{x}^R_0\in\Sigma$, 
\[
    \opdist(\bv{x}^{(\lambda)}(t),\Sigma)
    =
    \opdist(\bv{y}^\lambda,\Sigma)
    \le
    \|\bv{y}^\lambda - \bv{x}^R_0\|_{2},
    \quad
    t\ge0.
\]
Hence the property (\ref{item1.thm.main}) follows from \eqref{est5.prf.thm.main} combined with \eqref{est1.prf.thm.main}--\eqref{est4.prf.thm.main}.

Next we show (\ref{item2.thm.main}). For sufficiently small $\lambda\neq0$, we have by \eqref{def.F}, 
\[
    \delta V^\lambda \Big(1 - \frac{1}{R} v_0^\lambda\Big) 
    = 
    - \Omega^\lambda \overbar{v^{\perp,\lambda}w^\lambda_s}.
\]
Using \eqref{est1.prf.thm.main}--\eqref{est4.prf.thm.main}, 
\[
\begin{split}
    \overbar{v^{\perp,\lambda}w^\lambda_s}
    = 
    \lambda^2 
    \overbar{(\Phi_k)_1 \partial_s(\Phi_k)_2} 
    + o(\lambda^2). 
\end{split}
\]
Then, by the definition of $(\Phi_k)_1,(\Phi_k)_2$ in \eqref{def.Phi}, 
\[
\begin{split} 
    \overbar{(\Phi_k)_1 \partial_s(\Phi_k)_2} 
    &= 
    \frac{1}{2\pi R} 
    \int_{0}^{2\pi R}
    k \cos\Big(\frac{ks}{R}\Big) \frac{k}{R} \sqrt{k^2-1} \cos\Big(\frac{ks}{R}\Big)
    \dd s \\
    &= 
    \frac{k^2 \sqrt{k^2-1}}{2R}. 
\end{split}
\]
Substituting this back and noting $\Omega_k = \sqrt{k^2-1}/R^2$ gives, when $\lambda \to 0$, 
\[
    \delta V^\lambda 
    = 
    -  \left(\frac{k^2 (k^2-1)}{2R^3} \right)\lambda^2
    + o(\lambda^2).
\]
Hence we see from
\[
    \bv{x}^{(\lambda)}(s,t) - \bv{x}^R(s,t)
    =
    \del V^\lambda t \bv{e}_3 
    +Q^{\Omega^\lambda t}_z \bv{y}^\lambda(s + c^\lambda t)
    - \bv{x}^R_0(s)
\]
that, with a suitable choice of $\gamma,t_0>0$ and possibly after taking $\eps_{k}$ smaller,
\[
    |\bv{x}^{(\lambda)}(s,t) - \bv{x}^R(s,t)|
    \ge
    |\del V^\lambda |t 
    - 
    |Q^{\Omega^\lambda t}_z \bv{y}^\lambda(s + c^\lambda t) \cdot \bv{e}_3|
    \ge
    \gamma t,
\]
for any $(s,t)\in \T\times[t_0,\infty)$. This completes the proof of Theorem \ref{thm.main}. 
\hfill\BOX

\end{document}